\documentclass[12pt]{article}

\usepackage[dvipdfmx]{graphicx}

\usepackage[utf8]{inputenc}
\usepackage{url}
\usepackage{amsmath}
\usepackage{amsthm}
\usepackage{amssymb}
\usepackage{here}
\usepackage{caption}

\theoremstyle{remark} 
\newtheorem*{cproof}{Proof of Claim}

\newcommand{\claimqed}{\hfill$\blacksquare$}

\newcommand{\ran}{\mathrm{ran}}
\newcommand{\ZFC}{\mathrm{ZFC}}
\newcommand{\GCH}{\mathrm{GCH}}
\newcommand{\cf}[1]{\mathrm{cf}\left({#1}\right)}

\newcommand{\lkakko}{\left\lbrace}
\newcommand{\rkakko}{\right\rbrace}

\newcommand{\Ord}{\mathrm{Ord}}

\newcommand{\imageof}{{}^{\,{\prime}{\prime}}}

\newcommand{\denincl}[1]{\mathrm{den}\left({#1}, \subseteq \right)}

\usepackage{xcolor}

\usepackage{tikz}

\usetikzlibrary{arrows}

\newcommand{\Stat}{\mathsf{stat}}
\newcommand{\club}{\mathsf{club}}

\newcommand{\otp}[1]{\mathrm{otp}\left({#1}\right)}

\newcommand{\Fn}{\mathrm{Fn}}

\newcommand{\pposet}{\mathbb{P}}

\newcommand{\ncol}[1]{\mathsf{Col}({#1})}
\newcommand{\nchr}[1]{\mathsf{Chr}({#1})}
\newcommand{\nlist}[1]{\mathsf{List}({#1})}
\newcommand{\nrlist}[1]{\mathsf{List}^{*}({#1})}

\newcommand{\col}[2]{\mathsf{Col}({#1}, {#2})}
\newcommand{\lst}[2]{\mathsf{List}({#1}, {#2})}
\newcommand{\rlist}[2]{\mathsf{RList}({#1}, {#2})}
\newcommand{\slist}[2]{\mathsf{SList}({#1}, {#2})}
\newcommand{\chr}[2]{\mathsf{Chr}({#1}, {#2})}
\usepackage{hyperref}
\usepackage{cleveref}

\theoremstyle{definition}
\newtheorem{defi}{Definition}[section]

\theoremstyle{plain}
\newtheorem{prop}[defi]{Proposition}
\newtheorem{theo}[defi]{Theorem}
\newtheorem{cor}[defi]{Corollary}
\newtheorem{lem}[defi]{Lemma}
\newtheorem{fac}[defi]{Fact}

\newtheorem*{cla}{Claim}
\newtheorem{que}{Question}

\crefname{prop}{Proposition}{Propositions}
\crefname{theo}{Theorem}{Theorems}
\crefname{cor}{Corollary}{Corollaries}
\crefname{lem}{Lemma}{Lemmata}
\crefname{fac}{Fact}{Facts}
\crefname{prob}{Problem}{Problems}
\crefname{que}{Question}{Questions}

\usepackage[backend=biber,style=numeric,maxcitenames=10,mincitenames=10,maxbibnames=1000,sorting=nty,isbn=false,url=false,eprint=true]{biblatex}
\usepackage{geometry}
\usepackage{here}
\usepackage[affil-it]{authblk}
\usepackage{diagbox}

\DeclareFieldFormat{doi}{%
  \mkbibacro{DOI}\addcolon\space
  \ifhyperref
    {\href{https://doi.org/#1}{\nolinkurl{https://doi.org/#1}}}
    {\nolinkurl{https://doi.org/#1}}}

\renewbibmacro{in:}{}
\title{Stationary list colorings}
\author{Yusuke Hayashi\thanks{Supported by JST SPRING, Japan Grant Number JPMJSP2148}}
\affil{Graduate School of System Informatics, Kobe University, 1-1 Rokkodai, Nada-ku, 657-8501 Kobe, Japan. E-mail: 219x504x@stu.kobe-u.ac.jp}
\date{\today}

\begin{document}
\maketitle

\begin{abstract}
    Komj\`{a}th \cite{komjath2013list} studied the list chromatic number of infinite graphs and introduced the notion of the restricted list chromatic number. For a graph $X = (V_X, E_X)$ and a cardinal $\kappa$, we say that $X$ is restricted list colorable for $\kappa$ ($\rlist{X}{\kappa}$) if for any $L:V_X \rightarrow [\kappa]^\kappa$, there is some choice function $c$ of $L$ such that $c(v) \neq c(w)$ whenever $\lkakko v, w \rkakko \in E_X$. In this paper, we discuss its variation, the stationary list colorability for $\kappa$ ($\slist{X}{\kappa}$), which is obtained by replacing $[\kappa]^\kappa$ with the set of all stationary subsets of $\kappa$. 
    
    We compare the stationary list colorability with other coloring properties. Among other things, we prove that the stationary list colorability is essentially different from other coloring properties including the restricted list colorability. We also prove the consistency of that $\rlist{X}{\kappa}$ and $\slist{X}{\kappa}$ do not imply $\rlist{X}{\lambda}$ and $\slist{X}{\lambda}$ for $\lambda > \kappa$.
\end{abstract}

\section{Introduction}
Graph coloring is one of the central topics in combinatorics. Finite graph coloring has been studied extensively, and several invariants have been investigated, including the chromatic number, the coloring number, and the list chromatic number.

On the other hand, many natural extensions have been considered for infinite graphs as well. In particular, Erd\"{o}s and Hajnal \cite{erdHos1966chromatic} studied the chromatic and the coloring numbers of infinite graphs widely. Let us recall those definitions. Let $X=(V_X, E_X)$ be an undirected simple graph, i.e., $V_X$ is a set and $E_X \subseteq [V_X]^2$. A function $c$ whose domain is $V_X$ is \textit{good} if, for all $v, w \in V_X$, $c(v) \neq c(w)$ holds whenever $\lkakko v, w \rkakko \in E_X$.

\begin{itemize}
    \item \textit{The coloring number $\ncol{X}$} is the least cardinal $\kappa$ such that there is a well-order $\prec$ of $V_X$ such that, for all $v \in V_X$, $\lkakko w \in V_X \mid w \prec v  \land \lkakko v, w \rkakko \in E_X \rkakko$ has size ${<} \kappa$. 
    \item \textit{The chromartic number $\nchr{X}$} is the least cardinal $\kappa$ such that there is a good coloring $c:V_X \rightarrow \kappa$.
\end{itemize}

Later, Komj\'{a}th \cite{komjath2013list} studied the list chromatic number of infinite graphs and introduced the notion of the restricted list chromatic number:

\begin{itemize}
    \item \textit{The list chromatic number $\nlist{X}$} is the least cardinal $\kappa$ such that, for every function $L:V_X \rightarrow [\Ord]^\kappa$, there is some choice function $c$ of $L$ such that $c$ is good.
    \item \textit{The restricted list chromatic number $\nrlist{X}$} is the least cardinal $\kappa$ such that, for every function $L:V_X \rightarrow [\kappa]^\kappa$, there is some choice function $c$ of $L$ such that $c$ is good. 
\end{itemize}

Komj\'{a}th \cite{komjath2013list} investigated these notions and showed the consistency of $\nlist{X} = \ncol{X}$ for every infinite graph $X$ with $\ncol{X} \geq \aleph_0$ and that of $\nlist{X} = \aleph_0$ iff $\nchr{X} = \aleph_0$ for every infinite graph $X$ with $|V_X| = \aleph_1$. In this paper, we focus on the following result due to Komj\`{a}th \cite{komjath2013list}:

\begin{fac}\label{fac:Komjath fact}
    Assume $\GCH$. Suppose that $\nlist{X}$ is infinite. Then $\ncol{X} \leq \nlist{X}^{+}$. Also, if $\nrlist{X}$ is regular, then $\ncol{X} \leq \nrlist{X}^{+}$.
\end{fac}

Motivated by these works, we revisit and extend his approach by introducing a new variant, called the stationary list coloring, where each vertex is assigned a stationary subset of a regular cardinal $\kappa$ as its list of colors.

In this paper, we study five coloring notions:
\begin{defi}
    Let $X=(V_X, E_X)$ be a graph and $\kappa$ be a cardinal.
    \begin{itemize}
        \item $\col{X}{\kappa}$ holds if there is a well-order $\prec$ of $V_X$ such that,for all $v \in V_X$, $\{ w \in V_X \mid w \prec v  \land \lkakko v, w \rkakko \in E_X \}$ has size ${<} \kappa$.
        \item $\lst{X}{\kappa}$ holds if for each $L: V_X \rightarrow [\Ord]^\kappa$ there is a choice function $c$ of $L$ such that $c$ is a good coloring.
        \item $\rlist{X}{\kappa}$ holds if for each $L: V_X \rightarrow [\kappa]^\kappa$ there is a choice function $c$ of $L$ such that $c$ is a good coloring.
        \item Suppose $\kappa$ is a regular cardinal. Let $\Stat_\kappa$ be the set of all stationary subsets of $\kappa$. $\slist{X}{\kappa}$ holds if for each $L: V_X \rightarrow \Stat_\kappa$ there is a choice function $c$ of $L$ such that $c$ is a good coloring.
        \item $\chr{X}{\kappa}$ holds if there is a good coloring $c:V_X \rightarrow \kappa$.
    \end{itemize}
\end{defi}

For regular cardinals $\kappa < \lambda$, we can summarize the relationships among these properties as shown in Figure 1. See \cite{komjath2013list} for $\col{X}{\kappa} \rightarrow \lst{X}{\kappa}$. $\lst{X}{\kappa} \rightarrow \rlist{X}{\kappa} \rightarrow \slist{X}{\kappa}$ since $\Stat_\kappa \subseteq [\kappa^\kappa] \subseteq [\Ord]^\kappa$. The other implications are immediate from the definitions.

\begin{figure}[H]
    \centering
    \caption{Diagram of variants of coloring properties}
    \label{fig:diagram of variants}
    \begin{tikzpicture}
            \node (chr1) at (2, 4) {$\chr{X}{\kappa}$};
            \node (chr0) at (2, 2) {$\chr{X}{\lambda}$};
            
            \node (sl1) at (-1, 4) {$\slist{X}{\kappa}$};
            \node (sl0) at (-1, 2) {$\slist{X}{\lambda}$};
            
            \node (rl1) at (-4, 4) {$\rlist{X}{\kappa}$};
            \node (rl0) at (-4, 2) {$\rlist{X}{\lambda}$};
            
            \node (l1) at (-7, 4) {$\lst{X}{\kappa}$};
            \node (l0) at (-7, 2) {$\lst{X}{\lambda}$};
            
            \node (col1) at (-10, 4) {$\col{X}{\kappa}$};
            \node (col0) at (-10, 2) {$\col{X}{\lambda}$};

            \draw[thick, ->] (col1) -- (l1);
            \draw[thick, ->] (l1) -- (rl1);
            \draw[thick, ->] (rl1) -- (sl1);
            \draw[thick, ->] (sl1) -- (chr1);
            
            \draw[thick, ->] (col0) -- (l0);
            \draw[thick, ->] (l0) -- (rl0);
            \draw[thick, ->] (rl0) -- (sl0);
            \draw[thick, ->] (sl0) -- (chr0);
            
            \draw[thick, ->] (col1) -- (col0);
            
            \draw[thick, ->] (l1) -- (l0);
            
            \draw[thick, ->] (chr1) -- (chr0);  
        \end{tikzpicture}
\end{figure}
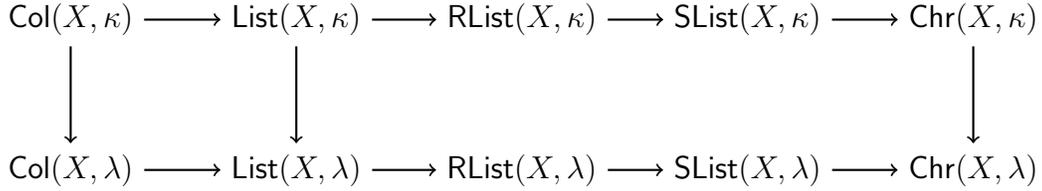

Then the following questions arise naturally:
\begin{enumerate}
    \item Is $\slist{X}{\kappa}$ essentially different from $\rlist{X}{\kappa}$ and $\chr{X}{\kappa}$?
    \item Can we prove an analogue of \Cref{fac:Komjath fact} for $\slist{X}{\kappa}$?
    \item Are $\rlist{X}{-}$ and $\slist{X}{-}$ monotonic? That is, if $\kappa < \lambda$, then does $\rlist{X}{\kappa}$ (respectively, $\slist{X}{\kappa}$) imply $\rlist{X}{\lambda}$ (respectively, $\slist{X}{\lambda}$)?
\end{enumerate}

In this paper, we answer these questions. 
\begin{defi}
    Let $\kappa$ and $\lambda$ be cardinals. The complete bipartite graph $K_{\kappa, \lambda}$ is defined as follows:
    \begin{itemize}
        \item $V_{K_{\kappa, \lambda}} = A\sqcup B$ where $A$ and $B$ are sets with $|A|=\kappa$ and $|B|=\lambda$.
        \item $E_{K_{\kappa, \lambda}} = \lkakko \{a, b\} \mid a \in A, b\in B \rkakko$.
    \end{itemize}
\end{defi}

Here, we can answer the questions above as follows:
\begin{theo}[see \Cref{theo:chr neq slist,theo:slist neq rlist,theo:main3,theo:main4}]
    Let $\kappa$ and $\lambda$ be regular cardinals.
    \begin{enumerate}
        \item We have the following:
        \begin{itemize}
            \item $\chr{K_{2^\kappa, 2^\kappa}}{\kappa}$ holds but $\slist{K_{2^\kappa, 2^\kappa}}{\kappa}$ fails.
            \item $\slist{K_{2^\kappa, \kappa}}{\kappa}$ holds but $\rlist{K_{2^\kappa, \kappa}}{\kappa}$ fails.
        \end{itemize}
        \item Assume $\GCH$. If $\slist{X}{\kappa}$ holds, then
        \begin{itemize}
            \item $\col{X}{\kappa^{++}}$ holds;
            \item moreover, if $\lambda \in I[\lambda]$ for all regular cardinals $\lambda$, then $\col{X}{\kappa^+}$ holds.
        \end{itemize}
        \item Assume $\GCH$. Let $\kappa < \lambda$ be regular cardinals and $\mu> \lambda$ be a singular cardinal with cofinality $\kappa$. Define $\pposet = \Fn_{<\kappa}(\mu, 2)$. Let $G$ be a $\pposet$-generic filter over $V$. Then in $V[G]$, 
            \begin{itemize}
                \item $\rlist{K_{\mu, \mu}}{\kappa}$ holds but $\rlist{K_{\mu, \mu}}{\lambda}$ fails, and
                \item $\slist{K_{\mu, \mu}}{\kappa}$ holds but $\slist{K_{\mu, \mu}}{\lambda}$ fails.
            \end{itemize}
        In other words, in $V[G]$, the monotonicities of $\mathsf{RList}$ and $\mathsf{SList}$ fail simultaneously.
    \end{enumerate}
\end{theo}

The paper is organized as follows. In Section 2 we analyze the properties of stationary list colorings and prove that they are distinct from both the restricted and the ordinary list colorings. We also discuss an analogue of Komj\'{a}th's results. In Section 3 we present consistency results showing the failure of monotonicity for the stationary/restricted list‐coloring properties. In Section 4, we conclude with several open questions and possible extensions of our results.

\section*{Acknowledgements}
    I would like to express my deepest gratitude to my supervisor Hiroshi Sakai for his patient supervision and inspiring discussions on my work. Without his guidance and help, this paper would have been difficult to complete.
    
    This work was supported by JST SPRING, Japan Grant Number JPMJSP2148.

\section{Analysis of stationary list coloring}
First, we show that $\slist{X}{\kappa}$ is different from $\chr{X}{\kappa}$. For later use, we state the claim in a more general form.

\begin{defi}
    Let $X$ be a set. The density of $X$ ($\denincl{X}$) is the least cardinality of a dense subset of $X$. That is, $\denincl{X} \leq \kappa$ if there is some $Y \subseteq X$ such that
    \begin{itemize}
        \item $|Y| \leq \kappa$, and
        \item for every $x \in X$, there is some $y \in Y$ such that $y \subseteq x$.
    \end{itemize}
    
\end{defi}

\begin{prop}\label{prop:nes for slist}
    Let $\lambda$ be a cardinal, and $X=K_{\lambda, \lambda}$. For a regular cardinal $\kappa\geq \lambda$, if $\lambda \geq \denincl{\Stat_\kappa}$ holds, then $\slist{X}{\kappa}$ fails.
\end{prop}
\begin{proof}
    Let $A, B$ be the two parts of $V_X$ with $|A| = |B| = \lambda$. Let $\mathcal{S}\subseteq \Stat_\kappa$ be a dense set of stationary sets in $\kappa$ with $|\mathcal{S}| \leq \lambda$. Take a list function $L$ with $L\imageof A = L\imageof B = \mathcal{S}$. Let $c$ be a choice function of $L$. It suffices to show that $c$ is not good. Since $\mathcal{S}$ is dense in $(\Stat_\kappa, \subseteq)$, $c\imageof A$ intersects all stationary sets in $\kappa$, that is, $c\imageof A$ contains a club subset in $\kappa$. Then there is some $S \in \mathcal{S}$ such that $S \subseteq c\imageof A$ since $\mathcal{S}$ is dense in $(\Stat_\kappa, \subseteq)$ and $\club_\kappa \subseteq \Stat_\kappa$. Let $w \in B$ with $L(w)=S$. Then $c(w) \in L(w) = S \subseteq c\imageof A$. Therefore we can find a vertex $v \in A$ such that $c(v) = c(w)$. By the completeness of $X$, $\lkakko v, w \rkakko \in E_X$ and this implies that $c$ is not a good coloring.
\end{proof}

\begin{cor}\label{theo:chr neq slist}
    $\chr{K_{2^\kappa, 2^\kappa}}{\kappa}$ holds but $\slist{K_{2^\kappa, 2^\kappa}}{\kappa}$ fails.
\end{cor}
\begin{proof}
    Let $X = K_{2^\kappa, 2^\kappa}$. Since $X$ is a bipartite graph, even $\chr{X}{2}$ holds. Here $2^\kappa \geq |\Stat_\kappa| \geq  \denincl{\Stat_\kappa}$. Then, by \Cref{prop:nes for slist}, $\slist{X}{\kappa}$ fails.
\end{proof}

The following show the difference between $\slist{X}{\kappa}$ and $\rlist{X}{\kappa}$.

\begin{theo}\label{theo:slist neq rlist}
    $\slist{K_{2^\kappa, \kappa}}{\kappa}$ holds but $\rlist{K_{2^\kappa, \kappa}}{\kappa}$ fails.
\end{theo}
\begin{proof}
    Let $X= K_{2^\kappa, \kappa}$, and $V_X = A \sqcup B$ be the bipartite with $|A|=\kappa$ and $|B|=2^\kappa$. 
    
    First, we show that $\slist{X}{\kappa}$ holds. To prove this, take an arbitrary list function $L:V_X \rightarrow \Stat_\kappa$ and construct a good coloring $c$ choosing from $L$. We identify $A$ with $\kappa$. Since $\kappa$ is regular, we can choose $\langle c(i) \in L(i) \mid i <\kappa \rangle$ and $\langle \gamma_i < \kappa \mid i < \kappa \rangle$ as follows:
    \begin{itemize}
        \item For any $i<\kappa$, $\gamma_i < c(i) < \gamma_{i + 1}$.
        \item If $i < \kappa$ is a limit ordinal, then $\sup_{j < i} \gamma_j = \gamma_i$.
    \end{itemize}
    Then $\lkakko \gamma_i \mid i < \kappa\rkakko$ is a club in $\kappa$ and $\lkakko c(i) \mid i<\kappa \rkakko \cap \lkakko \gamma_i\mid i<\kappa \rkakko = \emptyset$. Since $L(v)$ is stationary in $\kappa$ for every $v\in V_X$, $L(v)\cap \lkakko \gamma_i \mid i < \kappa \rkakko \neq \emptyset$, and we can take $c(v) \in L(v) \cap \lkakko \gamma_i \mid i <\kappa \rkakko$ for all $v\in B$. Thus $c$ is a good coloring.
    
    Next, we prove that $\rlist{X}{\kappa}$ fails. To prove this, we construct a list function $L:V_X\rightarrow [\kappa]^\kappa$ such that for any choice function $c$ of $L$, $c$ is not good. We identify $A$ and $B$ with $\kappa$ and $[\kappa]^\kappa$, respectively. Define the list function $L$ by
    \[ L(v) = \begin{cases}
        \kappa\setminus v &\text{if $v\in A = \kappa$} \\
        v & \text{if $v \in B = [\kappa]^\kappa$}
    \end{cases}\]
    Let $c$ be a choice function of $L$, and we show that $c$ is not good. Let $w = c \imageof A$. By the definition of $c$ and the regularity of $\kappa$, $w  \in [\kappa]^\kappa = B$. Since $c(w)\in L(w)= w = c\imageof A$, there is some $v\in A$ such that $c(w)= c(v)$. By the completeness of $X$, $\lkakko v, w\rkakko\in E_X$. Therefore $c$ is not a good coloring.
\end{proof}

Next, we prove an analogue of \Cref{fac:Komjath fact}. More concretely, we prove the following.
    
\begin{theo}\label{theo:main3}
    Assume $\GCH$. Let $\kappa$ be a regular cardinal. If $\slist{X}{\kappa}$, then
        \begin{enumerate}
            \item $\col{X}{\kappa^{++}}$ holds;
            \item moreover, if $\lambda \in I[\lambda]$ for all regular cardinals $\lambda$, then $\col{X}{\kappa^+}$.
        \end{enumerate}
\end{theo}

By \Cref{theo:slist neq rlist}, $\slist{X}{\kappa}$ does not imply $\col{X}{\kappa}$, and even $\rlist{X}{\kappa}$ in general. Therefore, at least second part of \Cref{theo:main3} is optimal.

In order to prove \Cref{theo:main3}, we need to prepare several properties and facts.

First, we recall the ideal $I[\lambda]$ for a regular cardinal $\lambda$. 

    \begin{defi}
        Define $I[\lambda]$ by the set of all $S \in \mathcal{P}(\lambda)$ such that there exist a sequence $\langle b_\xi \mid \xi < \lambda \rangle$ and a club $C\subseteq \lambda$ such that
        \begin{itemize}
            \item $b_\xi$ is a bounded subset of $\lambda$ and 
            \item for all $\gamma \in C\cap S$, there is some $c_\gamma\subseteq\gamma$ such that 
                \begin{itemize}
                    \item $c_\gamma$ is cofinal in $\gamma$,
                    \item  $\otp{c_\gamma} = \cf{\gamma}$, and
                    \item for all $\delta < \gamma$, there is some $\eta < \gamma$ such that $c_\gamma \cap \delta = b_\eta$.
                \end{itemize}
        \end{itemize}
    \end{defi}

    In order to prove \Cref{fac:Komjath fact}, Komj\'{a}th \cite{komjath2013list} introduced some specific graphs as follows. Here, let $N(v)$ denote the set of all neighborhoods of $v$ for $v \in V_X$. That is, $N(v) = \lkakko w \in V_X \mid \lkakko v, w \rkakko \in E_X \rkakko$.

    \begin{defi}\label{defi:type graph}
    Let $\kappa$ be a cardinal.
        \begin{enumerate}
            \item we say that a bipartite graph $X$ is \textit{$\kappa$-type 1} if for some cardinal $\lambda \geq \kappa$, we have
            \begin{enumerate}
                \item the bipartition $(A, B)$ satisfies $|A| = \lambda$ and $|B|= \lambda^+$, and
                \item $|N(b)|\geq \kappa$ holds for all $b \in B$.
            \end{enumerate}
            \item we say that a graph $X$ is \textit{$\kappa$-type 2} if there is some regular cardinal $\lambda > \kappa$, $X$ is isomorphic to the following graph $Y$:
            \begin{enumerate}
                \item $V_Y = \lambda$, and
                \item $T_Y = \lkakko \xi \in E_{\cf{\kappa}}^\lambda \mid \otp{N(\xi) \cap \xi} = \kappa \land N(\xi)\cap \xi\text{ is cofinal in } \xi  \rkakko$ is stationary in $\lambda$.
            \end{enumerate}
        \end{enumerate}
    \end{defi}

    In \cite{komjath2013list}, Komj\'{a}th showed some necessary and sufficient condition for $\col{X}{\kappa}$.

    \begin{fac}\label{fac:characterization}
        Let $X$ be an infinite graph and $\kappa$ be an infinite cardinal.
        \begin{itemize}
            \item If $\col{X}{\kappa}$ fails, then $X$ has a subgraph $Y$ such that $Y$ is $\kappa$-type 1 or 2.
            \item If $\col{X}{\kappa^+}$ fails, then $X$ has a subgraph $Y$ such that $Y$ is $\kappa$-type 1.
        \end{itemize}
    \end{fac}

    In the proof of \Cref{fac:Komjath fact}, Komj\'{a}th actually proved the contrapositive. More concretely, in \cite{komjath2013list}, the following is proved.
    
    \begin{fac}\label{fac:type1 implies lst and rlst fails}
    Assume $\GCH$. Let $\kappa$ be a cardinal and $X$ be a $\kappa$-type 1 graph. Then,
    \begin{itemize}
        \item $\lst{X}{\kappa}$ fails.
        \item If $\kappa$ is regular, then $\rlist{X}{\kappa}$ fails.
    \end{itemize}
    \end{fac}
    
    Combining \Cref{fac:characterization,fac:type1 implies lst and rlst fails}, we obtain that the failure of $\col{X}{\kappa^{+}}$ for a cardinal (regular cardinal) $\kappa$ implies the failure of $\lst{X}{\kappa}$ ($\rlist{X}{\kappa}$), respectively. This is the flow of the Komj\'{a}th proof.
    
    Our strategy for the proof of \Cref{theo:main3} is based on this argument. That is, we show that $\slist{Y}{\kappa}$ fails for each $\kappa^{+}$-type 1 (or $\kappa^{++}$-type 2) graph $Y$ under $\GCH$ (and the approachability assumption), respectively.

    \begin{lem}\label{lem:case1}
        Assume $\GCH$. If $X$ is $\kappa^+$-type 1, then $\slist{X}{\kappa}$ fails.
    \end{lem}

    \begin{proof}
        Let $V_X = A\sqcup B$ be the bipartition of $V_X$ with $|A|=\lambda$ and $|B|=\lambda^+$ for some cardinal $\lambda$. 

        We construct a list function $L:A\sqcup B \rightarrow \Stat_\kappa$ that has no good coloring function. We may assume that $A=\lambda$. By taking a subgraph of $X$, if necessary, we can ensure that this $\lambda$ is minimal with respect to the conditions in \Cref{defi:type graph}. That is, if $A' \subseteq A$, $B'\subseteq B$, $|A'|^+ = |B'|$, and $|N(b)|\geq \kappa^+$ for all $b\in B$, then $|A'| = \lambda$. Also, by shrinking the edges, we may assume that $|N(b)| = \kappa^+$ for all $b\in B$.
        \begin{cla}
            $\cf{\lambda} = \kappa^+$ holds. 
        \end{cla}
        \begin{cproof}
            For a contradiction, suppose that $\cf{\lambda} \neq \kappa^+$. Note that $\kappa^{+} < \lambda$. For each $b \in B$, there is some $\xi_b < \lambda$ such that $|N(b)\cap \xi_b|=\kappa^+$. By the pigeonhole principle, there is some $\xi < \lambda$ such that $B_0 = \lkakko b \in B \mid \xi_b = \xi \rkakko$ has size $\lambda^+$. Since $\GCH$ holds, $|[\xi]^{\kappa^+}| \leq \lambda$. Therefore we can take $B_1\subseteq B_0$ with $|B_1| = \lambda^{+}$ such that $A'=N(b) \cap \xi$ is constant on $\xi \in B_1$. Let $B' \subseteq B_1$ with $|B'| = \kappa^{++}$. Then, $X\upharpoonright (A' \sqcup B')$ is the complete bipartite subgraph and $\kappa^{++} = |A'|^+ = |B'|$. However, this contradicts the minimality of $\lambda$ since $\kappa^+ < \lambda$.\claimqed
        \end{cproof}

         Take a cofinal sequence $\langle \lambda_\sigma\mid \sigma < \kappa^+ \rangle$ with $\lambda_0 = 0$.
        
        Since $\GCH$ holds, there is a sequence $\langle C_\sigma \mid \sigma< \kappa^+ \rangle$ of clubs in $\kappa$ such that
        \begin{itemize}
            \item if $\sigma \leq \tau$, then $C_\sigma \subseteq^{\ast} C_\tau$, and
            \item for each club $C$, there is some $\sigma < \kappa^+$ such that $C_\tau \subseteq^{\ast} C$ for all $\tau \geq \sigma$,
        \end{itemize}
        where $S\subseteq^* S'$ means $S\setminus S'$ is bounded in $\kappa$ for $S, S' \in \mathcal{P}(\kappa)$. 
        
        For $\xi \in [\lambda_\sigma, \lambda_{\sigma + 1} ) \subseteq \lambda = A$, we define $F(\xi) = C_\sigma$.
        \begin{cla}
            For $b \in B$, there is an $i_b < \kappa$ such that, for every club $C$ in $\kappa$, there exists some $\xi \in N(b)$ that $F(\xi) \setminus (i_b + 1) \subseteq C$ holds.
        \end{cla}
        \begin{cproof}
            Take an arbitrary $b \in B$. Suppose, for contradiction, that there is some $C_i \subseteq \kappa$ for each $i < \kappa$ such that
            \begin{itemize}
                \item $C_i$ is a club in $\kappa$, and
                \item $F(\xi) \setminus (i + 1) \nsubseteq C_i$ for every $\xi \in N(b)$.
            \end{itemize}
            Let $D = \bigtriangleup _{i < \kappa} C_i$. Since $D$ is a club in $\kappa$, there is some $\xi \in N(b)$ and $i < \kappa$ such that $F(\xi) \setminus (i + 1) \subseteq D$. By the definition of $D$, $D\setminus (i + 1) \subseteq C_i$. Therefore $F(\xi) \setminus (i + 1) \subseteq D \setminus (i + 1) \subseteq C_i$ and this contradicts for the choice of $C_i$. \claimqed  
        \end{cproof}

        By the pigeonhole principle, there exists an $i < \kappa$ such that $B' = \lkakko b\in B \mid i_b = i\rkakko$ has size $\lambda^+$. We define the list function $L$ as follows. For $\xi \in A$, we define $L(\xi) = F(\xi) \setminus (i + 1)$. Since $\GCH$ holds, we can identify $B'$ with the set $\prod_{\xi \in A} L(\xi)$. Here, for $g \in B'$, define $L(g) = \lkakko g(\xi) \mid \xi \in N(g) \rkakko$ and for $b \in B \setminus B'$, define $L(b)=\kappa$. We show that $L$ is as desired. To prove this, we have to check that $L$ has no good coloring function and $L(g)$ is stationary in $\kappa$ for all $g \in B'$.

        First, we show that $L$ does not have a good coloring. Let $c$ be a choice function of $L$ and $g = c \upharpoonright A$. Then, $g  \in \prod_{\xi \in A} L(\xi) = B'$. By the definition of $L$, $c(g) \in L(g) = \lkakko g(\xi) \mid \xi \in N(g) \rkakko$. Therefore there is some $\xi \in N(g)$ such that $c(g) =g(\xi)= c(\xi)$. Since $\xi \in N(g)$, $c$ is not good. 

        Next, we prove that $L(g)$ is stationary in $\kappa$ for all $g \in B'$. Take an arbitrary $g \in B'$ and a club set $C$ in $\kappa$. It suffices to show that $L(g) \cap C \neq \emptyset$. By the choice of $i$ and $B'$, there is a $\xi \in N(g)$ such that $L(\xi) = F(\xi) \setminus (i + 1) \subseteq C$. Therefore $g(\xi) \in L(g) \cap C \neq \emptyset$, so we are done.
    \end{proof}

    Now we can prove the first part of \Cref{theo:main3}

    \begin{proof}[Proof of 1. in \Cref{theo:main3}]
        Let $X$ be an infinite graph. We show the contrapositive, that is, we assume the failure of $\col{X}{\kappa^{++}}$ and show the negation of $\slist{X}{\kappa}$. By the second part of \Cref{fac:characterization}, $X$ has a subgraph $Y$ such that $Y$ is $\kappa^+$-type 1. Then, by \Cref{lem:case1}, $\slist{Y}{\kappa}$ fails. This completes the proof.
    \end{proof}
    
    Next, we prove the second part of \Cref{theo:main3}. In order to prove this, by the first part of  \Cref{lem:case1}, it suffices to show that for every $\kappa^+$-type 2 graph $X$, $\slist{X}{\kappa}$ fails assuming that $\GCH$ and $\lambda \in I[\lambda]$ for all $\lambda$.

    Take such a graph $X$, and let $V_X = \lambda$ and $T_X$ be the set defined in \Cref{defi:type graph}.

    \begin{lem}\label{lem:elementary substructure}
        Assume $\GCH$. Suppose that for all sufficiently large regular cardinals $\theta$, there is an elementary substructure $M\prec \langle \mathcal{H}_\theta, \in, X, \kappa \rangle$ with $|M|<\lambda$ such that 
            \begin{itemize}
                \item $M\cap \lambda \in T_X$, and
                \item there is some $x\subseteq N(M\cap \lambda)\cap (M\cap \lambda)$ with $|x|=\kappa^+$ such that $x\cap \gamma\in M$ for all $\gamma < M\cap \lambda$.
            \end{itemize}
            
        Then, $\slist{X}{\kappa}$ fails.
    \end{lem}
    \begin{proof}
        Let $M$ and $x$ be witnesses and $\xi^* = M\cap \lambda$. For $\sigma < \kappa^+$, we define $x_\sigma$ by the initial segment of $x$ with $\otp{x_\sigma} = \sigma$, and $S_\sigma$ by the set of all $\xi \in T_X \setminus (\sup(x_\sigma) + 1)$ such that $x_\sigma\subseteq N(\xi)$ holds. Since $X, x_\sigma \in M$, $S_\sigma \in M$. Then for each $\sigma < \kappa^+$, $S_\sigma$ is a stationary subset of $\lambda$ since $\xi^* \in S_\sigma$. Since $\langle S_\sigma\mid \sigma < \kappa^+ \rangle$ is a $\subseteq$-decreasing continuous sequence of stationary sets with $\bigcap_{\sigma < \kappa^+} S_\sigma = \lkakko \xi^* \rkakko$, we can take $A\in [\kappa^+]^{\kappa^+}$ such that $T_\sigma = (S_{\sigma}\setminus \bigcup_{\tau < \sigma} S_\tau) \setminus \xi^*$ is unbounded in $\lambda$ for all $\sigma \in A$.

        Let $\club_\kappa$ be the set of all club subsets in $\kappa$. Take a list function $L:V_X \rightarrow \Stat_\kappa$ satisfying $L\imageof x = \club_\kappa$, $L\imageof T_\sigma = \Stat_\kappa$ for each $\sigma \in A$, and the other vertices sends $\kappa$. We show that $L$ is as desired. In order to prove this, take an arbitrary choice function $c$ of $L$ and show that $c$ is not good. 

        Since $L\imageof x = \club_\kappa$, $c\imageof x$ intersects every club set of $\kappa$, that is, $c\imageof x$ is stationary. Then, since $\otp{x} = \kappa^+$, there is some $\sigma \in A$ such that $c\imageof x_\sigma = c\imageof x$. By the definition of $c$, there is some $\xi \in T_\sigma$ such that $L(\xi) = c\imageof x_\sigma$. Thus there is an $\eta \in x_\sigma$ such that $c(\eta) = c(\xi)$. Since $\xi \in T_\sigma \subseteq  S_{\sigma}$, $\lkakko \eta, \xi \rkakko \in E_X$. Therefore $c$ is not a good function.
    \end{proof}

    \begin{lem}\label{lem:case 2 with appr assumption}
        Assume $\GCH$. Suppose that $\lambda \in I[\lambda]$ for every cardinal $\lambda$. If $X$ is $\kappa^+$-type 2, then $\slist{X}{\kappa}$ fails.
    \end{lem}

    \begin{proof}
        Let $V_X = \lambda$ and $T_X \subseteq E^\lambda _{\kappa^+}$ be the set defined in \Cref{defi:type graph}.

        \begin{itemize}
            \item Case 1: $\lambda$ is a limit cardinal, or $\lambda = \zeta^+$ for some cardinal $\zeta$ with $\cf{\zeta}\neq \kappa$. 
                
            It suffices to show that the assumptions in \Cref{lem:elementary substructure} holds. Take a sufficiently large regular cardinal $\theta$. By the cardinal arithmetic assumptions, we can find $M \prec \langle \mathcal{H}_\theta, \in, X, \kappa \rangle$ with $|M| < \lambda$, $M^{\kappa} \subseteq M$, and $M\cap \lambda \in T_X$. Let $\xi^* = M\cap \lambda$. Note that $\cf{\xi^*} = \kappa^+$ and $N(\xi^*)$ is cofinal in $\xi^*$ with $\otp{N(\xi^*)} = \kappa^+$. We show that $x = N(\xi^*)$ is a witness. For $\gamma < \xi^*$, $x \cap \gamma \in M$ holds since $M^{\kappa} \subseteq M$ and $|x \cap \gamma| \leq \kappa$. Thus, we are done.

            \item Case 2: $\lambda = \zeta^+$ for some $\zeta$ with $\cf{\zeta} = \kappa$.

            It suffices to show that the assumptions in \Cref{lem:elementary substructure} hold. In order to prove this, take a sufficiently large regular cardinal $\theta$ and find witnesses $M \prec \langle \mathcal{H}_\theta, \in, X, \kappa \rangle$ and $x\subseteq N(M\cap \lambda)$.

            Since $\lambda \in I[\lambda]$, we can take $\langle b_\xi\mid  \xi < \lambda \rangle$ and $C\subseteq \lambda$ such that

            \begin{itemize}
            \item $b_\xi$ is a bounded subset of $\lambda$ and 
            \item for all $\gamma \in C$, there is some $c_\gamma\subseteq\gamma$ such that 
                \begin{itemize}
                    \item $c_\gamma$ is cofinal in $\gamma$,
                    \item  $\otp{c_\gamma} = \cf{\gamma}$, and
                    \item for all $\delta < \gamma$, there is some $\eta < \gamma$ such that $c_\gamma \cap \delta = b_\eta$.
                \end{itemize}
            \end{itemize}

            Take an elementary substructure $M$ of $\langle \mathcal{H}_\theta, \in, X, \kappa, \zeta\rangle$ such that $|M| < \lambda$, $\langle b_\xi\mid  \xi < \lambda \rangle, C\in M$, $\zeta \subseteq M$, and $M\cap \lambda \in T_X\subseteq E^\lambda _{\kappa^+}$. Let $\xi^* = M\cap \lambda$. Note that $\cf{\xi^*} = \kappa^+$

            Since $C\in M$ and $C$ is a club, $C \cap \xi^*$ is cofinal in $\xi^*$, and hence $\xi^* \in C$. Since $\lambda \in I[\lambda]$, there is some cofinal subset $c\subseteq \xi^*$ such that $\otp{c} = \cf{\xi^*} = \kappa^+$ and for all $\delta < \xi^*$, there is some $\eta < \xi^*$ such that $c\cap \delta = b_\eta$. Since $\langle b_\xi \mid \xi < \lambda \rangle \in M$ and $\eta < \xi ^* = M \cap \lambda$, $c\cap \delta = b_\eta \in M$ for all $\delta < \xi^*$. Take such a cofinal subset $c$.

             Let $\langle \pi_\xi\mid \xi < \lambda\rangle$ be a sequence in $M$ such that $\pi_\xi$ is a surjection from $\zeta$ to $\xi$ for all $\xi < \lambda$. Take a sequence $\langle \zeta_i \mid i < \kappa \rangle \in M$ cofinal in $\zeta$ with $\zeta_0 > \kappa$. For every $i < \kappa$, define $X_i$ as follows:
            \[ X_i= \lkakko \eta < \xi^* \mid \pi_{\min (c\setminus (\eta + 1))}\imageof\zeta_i \ni \eta \rkakko \]
            Since $\cf{\xi^*} = \kappa^+$ and $\langle X_i \mid i < \kappa\rangle$ is an $\subseteq$-increasing sequence such that $\bigcup_{i < \kappa} X_i = \xi^*$, we can find an $i < \kappa$ such that $X_i \cap N(\xi^*)$ is unbounded in $\xi^*$. Take such an $i < \kappa$ and show that $x = X_i \cap N(\xi^*)$ is as desired. That is, we show that $x \cap \gamma \in M$ for all $\gamma < \xi^*$
            
            For every $\gamma < \xi^*$, $X_i\cap \gamma = \lkakko \eta < \gamma \mid \pi_{\min (c\setminus (\eta + 1))}\imageof\zeta_i \ni \eta \rkakko$ can be defined by using parameters in $M$ since $\min (c\setminus (\eta + 1)) \in M$, so $X_i \cap \gamma \in M$. Here, $|X_i \cap \gamma| \leq |X_i| \leq \zeta_i$. Therefore by using $\GCH$ and $\zeta \subseteq M$, $X_i\cap \gamma \cap N(\xi^*) \in M$ for all $\gamma < \xi^*$. 
        \end{itemize}
        In any case, we proved the negation of $\slist{X}{\kappa}$. This completes the proof.
    \end{proof}

    Combining \Cref{lem:case1,lem:case 2 with appr assumption}, we obtain the second one in \Cref{theo:main3}

    \begin{proof}[Proof of 2. in \Cref{theo:main3}]
        Let $X$ be an infinite graph. Suppose that 
        $\col{X}{\kappa^{+}}$ fails, and show the negation of $\slist{X}{\kappa}$. By the first part of \Cref{lem:case1}, $X$ has a subgraph $Y$ such that $Y$ is $\kappa^+$-type 1 or 2. If $Y$ is type 1, \Cref{lem:case1} implies the failure of $\slist{X}{\kappa}$. On the other hand, if $Y$ is type 2, \Cref{lem:case 2 with appr assumption} shows the negation of $\slist{X}{\kappa}$. This completes the proof.
    \end{proof}

\Cref{theo:main3} implies the following immediately.
\begin{cor}
    Assume $\GCH$. Let $\kappa$ be a regular cardinal. Then
        \begin{enumerate}
            \item $\slist{X}{\kappa}$ implies $\slist{X}{\kappa^{++}}$.
            \item If $\kappa \in I[\nu]$ for all cardinals $\nu$, then $\slist{X}{\kappa}$ implies $\slist{X}{\kappa^{+}}$.
        \end{enumerate}
\end{cor}

\section{Anti-monotonicity of coloring properties}
In this section, we show consistency results about anti-monotonicity of coloring properties. Our proof is based on the singular-cardinal approach of Raghavan–Shelah\cite{raghavan2020small} and Usuba\cite{usuba2025monotonicity}.
The aim of this section is to prove the following result.

\begin{theo}\label{theo:main4}
    Assume $\GCH$. Let $\kappa < \lambda$ be regular cardinals. Let $\mu> \lambda$ be a singular cardinal with cofinality $\kappa$. Define $\pposet = \Fn_{<\kappa}(\mu, 2)$. Let $G$ be a $\pposet$-generic filter over $V$. Then, in $V[G]$, $\rlist{K_{\mu, \mu}}{\kappa}$ holds but $\slist{K_{\mu, \mu}}{\lambda}$ fails.
\end{theo}

We first introduce a generalization of the reaping number $\mathfrak{r}$.

\begin{defi}
    Let $\kappa$ be a regular cardinal. 
    \begin{itemize}
        \item For $A, B \in [\kappa]^\kappa$, we say that $A$ $\kappa$-splits $B$ if $|B \cap A|=|B\setminus A| =\kappa$.
        \item For $\mathcal{R} \subseteq [\kappa]^\kappa$, we say that $\mathcal{R}$ is a $\kappa$-reaping family if for any $A \in [\kappa]^\kappa$ there is some $B \in \mathcal{R}$ such that $A$ does not $\kappa$-split $B$. 
        \item Define 
        \[ \mathfrak{r}_{\kappa} = \min\lkakko |\mathcal{R}| \mid \text{$\mathcal{R}$ is a $\kappa$-reaping family} \rkakko . \]
    \end{itemize}
\end{defi}

For example, $\mathfrak{r}$ is coincides with $\mathfrak{r}_{\omega}$. The next proposition states that $\lambda < \mathfrak{r}_{\kappa}$ is characterized precisely by $\rlist{K_{\lambda, \lambda}}{\kappa}$.

\begin{prop}\label{prop:characterization of rlist}
    Let $\lambda$ be a cardinal, and $X=K_{\lambda, \lambda}$. For a regular cardinal $\kappa\geq \lambda$, the following are equivalent.
    \begin{enumerate}
        \item $\lambda < \mathfrak{r}_{\kappa}$.
        \item $\rlist{X}{\kappa}$.
    \end{enumerate}
\end{prop}
\begin{proof}
    Let $A, B$ be the two parts of $V_X$ with $|A| = |B| = \lambda$.
    
    \noindent$(1\rightarrow 2)$ Let $L:A \sqcup B \rightarrow [\kappa]^\kappa$ be a list function. It suffices to construct a choice function $c$ of $L$ such that $c$ is good. Let $\mathcal{F} = \ran(L) \subseteq [\kappa]^\kappa$. Then $|\mathcal{F}| \leq \lambda < \mathfrak{r}_{\kappa}$. By the definition of $\mathfrak{r}_{\kappa}$, there is some $S \in [\kappa]^\kappa$ such that $S$ $\kappa$-splits all elements of $\mathcal{F}$. In particular, for all $v \in V_X$, $L(v)\cap S$ and $L(v) \setminus S$ are not empty. Define $c$ as follows:
    \[ c(v) \in \begin{cases}
        L(v) \cap S &\text{if $v \in A$, and} \\
        L(v) \setminus S &\text{if $v \in B$.} 
    \end{cases} \]
    Thus $c$ is a good coloring.

    \noindent$(2\rightarrow 1)$ We show the contrapositive, that is, we suppose $\lambda \geq \mathfrak{r}_\kappa$ and show the negation of $\rlist{X}{\kappa}$. Let $\mathcal{R}$ be a reaping family with $|\mathcal{R}| = \mathfrak{r}_{\kappa}$. We show that there is a list function $L: A\sqcup B \rightarrow [\kappa]^\kappa$ such that no choice function $c$ of $L$ is good. Take a list function $L$ satisfying $L\imageof A = L\imageof B = \lkakko X \setminus i \mid X \in \mathcal{R}, i < \kappa \rkakko$. Let $c$ be an arbitrary choice function of $L$. Let $S_0 = c\imageof A$ and $S_1 = c\imageof B$. By the definition of $L$, $|S_n \cap S| = \kappa$ for every $S \in \mathcal{R}$ and $n < 2$. Since $\mathcal{R}$ is a reaping family, we can take $S \in \mathcal{R}$ such that $|S \setminus S_0|< \kappa$. Since $|S \cap S_1| = \kappa$ and $S\cap S_1 \subseteq (S_0 \cap S_1) \cup (S\setminus S_0) \cap S_1$, $S_0 \cap S_1 \neq \emptyset$. Let $\gamma \in S_0 \cap S_1, v \in A,$ and $w \in B$ such that $c(v) =c(w) = \gamma$. Since $X$ is the complete bipartite graph, $\lkakko v, w \rkakko \in E_X$. Therefore $c$ is not good.
\end{proof}

\begin{proof}[Proof of \Cref{theo:main4}]
    Take a continuous increasing sequence $\langle \mu_i \mid i < \kappa \rangle$ cofinal in $\mu$. By \Cref{prop:characterization of rlist,prop:nes for slist}, it suffices to show that $V[G]\vDash$``$ \denincl{\Stat_\lambda} \leq \mu < \mathfrak{r}_\kappa$". Note that $\pposet$ is ${<}\kappa^+$-cc.
    
    \begin{cla}
        In $V[G]$, $\denincl{\Stat_\lambda} \leq \mu$ holds.
    \end{cla}
    \begin{cproof}
        By counting nice names, we can show that $|(\mathcal{P}(\lambda))^{V[G\upharpoonright\mu_i]}| \leq \mu_i$ holds. Thus it suffices to show that $\bigcup_{i < \kappa} (\mathcal{P}(\lambda))^{V[G\upharpoonright\mu_i]} \cap \Stat_{\lambda}^{V[G]}$ is dense in $(\Stat_\lambda, \subseteq)$ in $V[G]$.
        
        Let $\dot{S} \in V^\pposet$ be a nice name for a subset of $\lambda$ and $p \in \pposet$ such that $p\Vdash\text{``}\dot{S} \text{ is stationary in $\lambda$"}$. We need to find a condition $q \leq p$, $i < \kappa$, and a name $\dot{T} \in V^{\Fn_{<\kappa}(\mu_i, 2)}$ such that $q \Vdash \text{``} \dot{T} \text{ is stationary and } \dot{T} \subseteq \dot{S} \text{"}$. Let $\dot{S} = \bigcup_{\xi < \lambda} \lkakko \check{\xi}\rkakko \times A_\xi$ where $A_\xi$ is an antichain in $\pposet$. For $ i < \kappa$, define $\dot{S_i} = \bigcup_{\xi < \lambda} \lkakko \check{\xi}\rkakko \times (A_\xi \cap \Fn_{<\kappa}(\mu_i, 2)) \in V^{\Fn_{<\kappa}(\mu_i, 2)}$. Since $p\Vdash \text{``}\dot{S} = \bigcup_{i < \kappa} \dot{S_i} \text{"}$ and $\kappa < \lambda$, there are some $q\leq p$ and $i < \kappa$ such that $q\Vdash\text{``}\dot{S_i} \text{ is stationary in $\lambda$"}$. $\dot{T} = \dot{S_i}$ is as desired. \claimqed
    \end{cproof}

    \begin{cla}
        In $V[G]$, $\mu < \mathfrak{r}_\kappa$ holds.
    \end{cla}
    \begin{cproof}
        Let $p \in \pposet$ and $\lkakko \dot{X_\alpha} \mid \alpha < \mu \rkakko \subseteq V^{\pposet}$ be a family of nice names for elements of $[\kappa]^\kappa$. We need to find a name $\dot{X}$ for an element of $[\kappa]^\kappa$ such that $p \Vdash\text{``}\dot{X} \text{ $\kappa$-splits all elements of $\lkakko \dot{X_\alpha} \mid \alpha < \mu \rkakko$"}$. Let $\dot{X_\alpha} = \bigcup_{i < \kappa} \lkakko \check{i}\rkakko \times A^\alpha_i$ where $A^\alpha_i$ is an antichain in $\pposet$. For $i < \kappa$, define $S_i = \bigcup\lkakko\mathrm{supp}(q) \mid q \in A_j^\alpha, j < \kappa , \alpha < \mu_i\rkakko$. Since $\pposet$ is ${<}\kappa^+$-cc, $|S_i| \leq \mu_i \cdot \kappa^+< \mu$. Thus we can take $\gamma_i<\mu$ such that $\gamma_i \notin S_i$ for $i < \kappa$. Let $\dot{X} = \lkakko (\check{i}, \{(\gamma_i,1)\}) \mid i < \kappa \rkakko\in V^\pposet$. We show that $\dot{X}$ is as desired. 

    Take $p' \leq p$, $\alpha < \mu$, and $i < \kappa$. It suffices to show that there is some $q_0, q_1 \leq p'$ and $j \geq i$ such that $q_0 \Vdash\text{``} \check{j} \in \dot{X}_\alpha\setminus \dot{X}\text{"}$, and $q_1 \Vdash\text{``}\check{j} \in \dot{X}_\alpha\cap \dot{X} \text{"}$. We can take $j\geq i$ such that $\mu_j > \alpha$, $p' \not \Vdash j \notin \dot{X}_\alpha$, and $\gamma_{j} \notin \mathrm{supp}(p')$ since $p \Vdash\dot{X}_\alpha \in [\kappa]^\kappa$ and $|\mathrm{supp}(p')| < \kappa$. Then there is some $a \in A_j^\alpha$ such that $a\parallel p'$. Since $\gamma_j \notin \mathrm{supp}(p') \cup S_i$, both $\{(\gamma_j, 0)\}$ and $\{(\gamma_j, 1)\}$ are compatible with $a$ and $p'$. Here
    \begin{align*}
        q_0 &= a \cup p' \cup \{(\gamma_j, 0)\}\Vdash \check{j} \in \dot{X_\alpha} \setminus \dot{X}, \text{ and} \\
        q_1 &= a \cup p' \cup \{(\gamma_j, 1)\}q_1 \Vdash \check{j} \in \dot{X_\alpha} \cap \dot{X}.
    \end{align*}
    Therefore, $p \Vdash\text{``}\dot{X} \text{ $\kappa$-splits $ \dot{X_\alpha}$"}$ for all $\alpha < \mu$. \claimqed
    \end{cproof}
    This completes the proof.
\end{proof}

Since $\rlist{X}{\nu}$ implies $\slist{X}{\nu}$ for any graph $X$ and regular cardinal $\nu$, \Cref{theo:main4} implies the following directly.

\begin{cor}
    Let $V[G]$ be a forcing extended model in \Cref{theo:main4}. Then in $V[G]$, 
    \begin{itemize}
        \item $\rlist{K_{\mu, \mu}}{\kappa}$ holds but $\rlist{K_{\mu, \mu}}{\lambda}$ fails, and
        \item $\slist{K_{\mu, \mu}}{\kappa}$ holds but $\slist{K_{\mu, \mu}}{\lambda}$ fails.
    \end{itemize}
    In other words, in $V[G]$, the monotonicities of $\rlist{K_{\mu, \mu}}{-}$ and $\rlist{K_{\mu, \mu}}{-}$ fail simultaneously.
\end{cor}

\section{Questions}
In \Cref{theo:chr neq slist,theo:slist neq rlist}, we constructed graphs witnessing the fact that $\slist{X}{\kappa}$ is different from $\rlist{X}{\kappa}$ and $\chr{X}{\kappa}$. On the other hand, in \cite{komjath2013list}, Komj{\'a}th proved that $\rlist{X}{\kappa}$ consistently differs from $\lst{X}{\kappa}$. Then the following question arises:

\begin{que}
    Can we prove that $\rlist{X}{\kappa}$ is different from $\lst{X}{\kappa}$ in $\ZFC$?
\end{que}

Next, we do not know whether the assumption of \Cref{theo:main3} is optimal or not. That is,
\begin{que}
    Under $\GCH$, can we prove that $\slist{X}{\kappa}$ implies $\col{X}{\kappa^+}$ without any assumption with respect to regular $\kappa$?
\end{que}

In Section 3, we constructed the model that satisfies $\rlist{X}{\kappa}$ and the negation of $\slist{X}{\lambda}$ for some graph $X$ and regular cardinals $\kappa < \lambda$. In this model, actually, we can prove that $\mathfrak{r}_\lambda$ is singular and $\mathfrak{r}_\lambda < \mathfrak{r}_\kappa$ by a method similar to that of \Cref{theo:main4}. Can we show this anti-monotonicity of reaping numbers by some small forcings? More concretely,

\begin{que}
    Is it consistent that $\mathfrak{r} = \omega_3 > \omega_2 = \mathfrak{r}_{\omega_1}$?
\end{que}

In $\ZFC$, $\mathfrak{r}_\mu \geq \mathfrak{r}_{\cf{\mu}}$ holds for singular $\mu$. Can we separate these invariants?

\begin{que}
    Is it consistent that $\mathfrak{r}_\mu > \mathfrak{r}_{\cf{\mu}}$ for singular $\mu$?
\end{que}

\Cref{prop:characterization of rlist} showed the characterization of the generalized reaping numbers via graph coloring properties. Can we find a similar correspondence for other cardinal invariants? That is,

\begin{que}
    For some cardinal invariant $\mathfrak{x}$ and its generalization $\mathfrak{x}_\kappa$ for every cardinal $\kappa$, does there exist a graph $X_\lambda$ for each cardinal $\lambda$ and a coloring property $P$ such that $P(\kappa, X_\lambda)$ is equivalent to $\mathfrak{x}_\kappa < \lambda$?
\end{que}

Stationary sets are positive sets with respect to the nonstationary ideal. Therefore, in the same way as our stationary list coloring, we can consider $I^+$-list coloring for an arbitrary ideal $I$.

\begin{que}
    What can be said about the corresponding $I^+$-list coloring property for an ideal $I$?
\end{que}

	\printbibliography
\end{document}